\documentclass[12pt]{amsart}
\usepackage{color}
\usepackage{amssymb}
\usepackage{graphicx}
\usepackage{ mathrsfs }

\let\emptyset\varnothing
 \renewcommand{\epsilon}{\varepsilon}

\newcommand{\Z}{{\mathbb Z}}
 
\newcommand{\R}{{\mathbb R}}
\newcommand{\C}{{\mathbb C}}

\newcommand{\cA}{{\mathcal A}}

\newtheorem{Corollary}{Corollary}

 \newtheorem{Theorem}{Theorem}
  \newtheorem{Lemma}{Lemma}

\theoremstyle{Remark}
\newtheorem{ex}{Example}

\theoremstyle{Remark}
\newtheorem{Remark}{Remark}

\newtheorem{Definition}{Definition}

\begin{document}

\title{The Resultant of Developed Systems of Laurent Polynomials}

\dedicatory{To the memory of Vladivir Igorevich Arnold}

\author{A. G. Khovanskii}
\address{Department of Mathematics, University of Toronto, Toronto,
Canada; Moscow Independent University, Moscow, Russia.}
\email{askold@math.toronto.ca}

\author{Leonid Monin}
\address{Department of Mathematics, University of Toronto, Toronto,
Canada}
\email{lmonin@math.toronto.ca}

\begin{abstract}
Let $R_\Delta (f_1,\ldots,f_{n+1})$ be the {\it $\Delta$-resultant} (see below) of  $(n+1)$-tuple of Laurent polynomials. We provide an algorithm for  computing $R_\Delta$ assuming that an $n$-tuple $(f_2,\dots,f_{n+1})$ is {\it developed} (see sec.6).  We provide a relation between the product of $f_1$  over roots of $f_2=\dots=f_{n+1}=0$ in $(\C^*)^n$ and the product of $f_2$ over roots of $f_1=f_3=\dots=f_{n+1}=0$ in $(\C^*)^n$ assuming that the $n$-tuple $(f_1f_2,f_3,\ldots,f_{n+1})$ is developed. If all  $n$-tuples contained  in $(f_1,\dots,f_{n+1})$ are developed we  provide a signed version of Poisson formula for $R_\Delta$. In our proofs we use  topological arguments and topological version of the Parshin reciprocity laws.
\end{abstract}

\thanks{The first author is partially supported by the Canadian Grant No. 156833-12.}

\date{\today}
\maketitle

\section{Introduction.}

To a Laurent polynomial $f$ in $n$ variables one associates its Newton polyhedron $\Delta(f)$ which is a convex lattice polyhedron in $\R^n$ (through all of this paper by polyhedron we will mean compact convex polyhedron with integer vertices). A system of $n$ equations $f_1=\dots=f_n=0$  in $(\Bbb C^*)^n$ is called {\it developed} if  (roughly speaking) their Newton polyhedra $\Delta(f_i)$ are located generically enough with respect to each other. The exact definition (see also sec.6) is as follows: a collection of $n$ polyhedra $\Delta_1,\ldots,\Delta_n\subset \R^n$ is called {\it developed} if for any covector $v\in (\R^n)^*$ there is $i$ such that on the polyhedron  $\Delta_i$ the inner product with $v$  attains its biggest value precisely at a vertex of $\Delta_i$.

A developed system resembles an equation in one unknown. A polynomial in one variable of degree $d$ has exactly $d$ roots counting with multiplicity. The number of roots in $(\Bbb C^*)^n$ counting with multiplicities of a developed system is always determined by the Bernstein-Koushnirenko formula (see \cite{B}) (if the system is not developed this formula holds only for generic systems  with fixed Newton polyhedra).

As in the one-dimensional case, one can explicitly compute the sum of values of any Laurent polynomial over the roots of a developed system \cite{GKh},\cite{GKh1} and the product of all of the roots of the system  regarded as elements in the group $(\Bbb C^*)^n$ \cite{Kh2}. These results can be proved topologically \cite{GKh1}, using the topological identity between certain homology cycles related to developed system (see sec.7), the Cauchy residues theorem, and a topological version of the  Parshin reciprocity laws (see sec.8).

 To an $(n+1)$-tuple $A = (\cA_1,\dots,\cA_{n+1})$ of finite subsets in  $\Z^n$ one associates the {\it $A$-resultant} $R_A$. It is a polynomial defined up to sign in the coefficients   of Laurent polynomials $f_1, \ldots , f_{n+1}$ whose supports belong  to $A_1,\dots,A_{n+1}$ respectevely.  The  $A$-resultant is equal to $\pm1$ if the codimension of the variety of consistent systems in the space of all systems with supports in $A$ is greater than 1. Otherwise, $R_A$ is a polynomial which vanishes on the variety of consistent systems and such that the degree of $R_A$  in the coefficients of the $i$-th polynomial is equal to the generic number of roots of the system $f_1=\ldots=\hat f_i=\ldots=f_{n+1}=0$ (in which the equation $f_i=0$ is removed).

 The notion of $A$-resultant was introduced and studied in \cite{GKZ} under the following assumption on $A$: the lattice generated by the differences $a-b$ for all couples $a,b\in \cA_i$ and all $0\leq i\leq n+1$ is $\Z^n$. Under this assumption  the resultant $R_A$  is an irreducible polynomial (which was used  in  a definition of $R_A$ in  \cite{GKZ}). Later in \cite{Est} and \cite{D'AS} it was shown  that in the general case (i.e when the differences from $\cA_i$'s do not generate the whole lattice) $R_A$ is some power of an irreducible polynomial. The power is equal to the generic number of roots of a corresponding consistent system (see sec.13 for more details).

  To an  $(n+1)$-tuple $\Delta=(\Delta_1,\dots,\Delta_{n+1})$ of Newton polyhedra one  associates the $(n+1)$-tuple $A_\Delta$  of finite subsets $(\Delta_1\cap \Z^n,\dots, \Delta_{n+1}\cap \Z^n)$ in $\Z^n$. We define the {\it $\Delta$-resultant} as $A$-resultant for $A=A_\Delta$. In the paper we  deal with $\Delta$-resultants only. If a property of  $\Delta$-resultant is a known property of $A$-resultants for $A=A_\Delta$  we refer to a paper where the property of $A$-resultants is proven (without mentioning that the paper deals with $A$-resultants and not with $\Delta$-resultants). Dealing with $\Delta$-resultants only we  lose nothing: $A$-resultants can be reduced to $\Delta$-resultants. One can check that $R_A(f_1,\dots,f_{n+1})$ for $A=(A_1,\dots,A_{n+1})$ is equal to $R_\Delta(f_1,\dots,f_{n+1})$ for $\Delta=(\Delta_1,\dots,\Delta_{n+1})$, where $\Delta_1,\dots,\Delta_{n+1}$ are the convex hulls of the sets $A_1,\dots,A_{n+1}$.

A collection $\Delta$ is called {\it $i$-developed}  if its subcollection  obtained by removing the polyhedron $\Delta_i$ is developed. Using the Poisson formula (see \cite{PSt}, \cite{D'AS} and sec.13.3) one can show that for $i$-developed~$\Delta$ the identity
\begin{equation}\label{i R}
 R_\Delta =\pm \Pi_\Delta^{[i]}M_i
 \end{equation}
 holds, where $\Pi_\Delta^{[i]}$ is the product  of $f_i$ over the common zeros in $(\C^*)^n$  of $f_j$, for $j \neq i$, and $M_i$ is an explicit monomial in the vertex coefficients (i.e. the coefficient of $f_j$ in front of a monomial corresponding to a vertex of $\Delta_j$) of all the Laurent polynomials $f_j$ with $j \neq i$.

We provide an explicit algorithm for computing the term $\Pi_\Delta^{[i]}$ using the summation formula over the roots of a  developed system (Corollary 8). Hence we get {\it an explicit algorithm for computing the resultant $R_\Delta$} for an $i$-developed collection~$\Delta$. This  algo\-rithm heavily uses the Poisson formula (1).

If $(n+1)$-tuple $\Delta$ is $i$-developed  and $j$-developed for some $i\neq j$  the identity
\begin{equation}\label{i j P}
 \Pi_\Delta^{[i]} =\Pi_\Delta^{[j]}M_{i,j}s_{i,j}
 \end{equation}
 holds, where $M_{i,j}$ is an explicit monomial
in the coefficients of Laurent polynomials $f_1,\dots, f_{n+1}$ and $s_{i,j}=(-1)^{f_{i,j}}$  is an explicitly defined sign (Corollary 4).

Our proof of the identity (\ref{i j P}) is topological. We use the topological identity between cycles related to a developed system (see sec.7) and a topological version of the Parshin reciprocity laws. The identity (\ref{i j P})   generalizes the formula from \cite{Kh2} for the product in $(\C^*)^n$ of all roots of a developed system of equations.

An $(n+1)$-tuple $\Delta$ is called {\it completely developed}  if it is $i$-developed for every $1\leq i \leq n+1$. For completely developed $\Delta$ the identity
\begin{equation}\label{n+1}
\Pi^{[1]}_\Delta M_1s_1=\ldots= \Pi^{[n+1]}_\Delta M_{n+1}s_{n+1}
\end{equation}
holds, where $(M_1,\dots,M_{n+1})$ and $\left(\Pi^{[1]}_\Delta,\dots,\Pi^{[n+1]}_\Delta\right)$ are monomials and products appearing in (\ref{i R}) and $(s_1,\dots, s_{n+1})$ is an $(n+1)$-tuple of signs such that $s_is_j=s_{i,j}$ where $s_{i,j}$ are the explicit signs from identity (\ref{i j P}).

Our proof of the identities (\ref{n+1}) uses the  identity (\ref{i j P}) and does not rely on the theory of resultants. Using one general fact from this theory (Theorem 19) one can see that the quantities in the identities (\ref{n+1}) are equal to the $\Delta$-resultant, i.e. are equal to $\pm R_\Delta$. Thus the identities (\ref{n+1}) can be considered as a signed version of the Poisson formula for completely developed systems.

To make the paper more accessible we first describe all of the results in the classical one dimensional case.

\subsection*{Acknowledgements} The authors would like to thank Benjamin Briggs and Kiumars Kaveh for their help with editing earlier versions of this paper and the referee for their valuable comments which helped to improve the paper.

\section {Resultants in dimension one}

\subsection {Sylvester's formula for resultant}

Let $P_1=a_0+\dots+a_kz^k$, $P_2=b_0+\dots+ b_n z^n$ be polynomials in one complex variable of degrees $\leq k$ and $\leq n$. Sylvester defined  resultants $R^{[1]}$ and $R^{[2]}$ which are equal up to a sign. $R^{[1]}$ and $R^{[2]}$ are defined as the determinant of the  $(n+k)\times(n+k)$  matrices  $M_1$, $M_2$ respectively, where:
\[
M_1(P_1,P_2)=
  \begin{pmatrix}
    a_k &a_{k-1}&\ldots& a_0 & 0  &  0 \\
      0 &\ddots &\ddots&\ldots& \ddots& 0 \\
    0&0& a_k &a_{k-1}&\ldots& a_0\\
       b_n &b_{n-1}&\ldots& b_0 & 0  &  0 \\
      0 &\ddots &\ddots&\ldots& \ddots& 0 \\
    0&0& b_n &b_{n-1}&\ldots& b_0\\
  \end{pmatrix}
\]
and $M_2(P_1,P_2) = M_1 (P_2,P_1)$.

One can see that: 1) $R^{[1]}=(-1)^{kn}R^{[2]}$; 2) resultants are polynomials in the coefficients of $P_1$ and $P_2$ with degrees $n$ and $k$ in coefficients of $P_1$ and $P_2$ respectively; 3) the polynomials $R^{[1]}$ and $R^{[2]}$ have integer coefficients; 4) the coefficient of the monomial $a_k^n b_0^k$ in  $R^{[1]}$ is $+1$ and, hence, the coefficient of the same monomial in $R^{[2]}$ is $(-1)^{kn}$, 5) the coefficients of $R^{[1]}$ and $R^{[2]}$ are coprime integers (this follows from 3) and 4)).

Under the assumption $a_n\neq 0$, $b_k\neq 0$ the resultant $R^{[1]}$ (the resultant $R^{[2]}$) is equal to zero   if and only if the polynomials $P_1$ and $P_2$ have common root. One can show  that the variety of pairs of polynomials $P_1$ and $P_2$ having a common root is an irreducible quasi projective variety $X$ (a simple proof of a multidimensional version  of this statement can be found in \cite{GKZ}). Let $D$ be an irreducible polynomial equal to zero on $X$. According to the previous statement   $R^{[1]}$ and $R^{[2]}$ are equal to the same power of $D$, multiplied by some coefficients.

The Sylvester resultants can be generalized for Laurent polynomials. Consider two Laurent polynomials
\begin{equation} \label{f1,f2}
f_1=a_kz^k+\dots +a_nz^n,\quad f_2= b_lz^l+\dots +b_mz^m,
\end{equation}
 on $\C^*$ whose Newton polyhedra belong to the segments $\Delta_1,\Delta_2$  defined by inequalities  $k\leq x\leq n$, $l\leq x\leq m$. With  $f_1, f_2$ let us associate the pair of polynomials $P_1,P_2$, where
 \begin{equation} \label{P1,P2}
P_1=z^{-k}f_1=a_k+\dots +a_nz^{n-k},\quad P_2=z^{-l}f_2= b_l+\dots +b_mz^{m-l}.
\end{equation}

For $\Delta=(\Delta_1,\Delta_2)$ we define  resultants $R^{[1]}_{\Delta}$, $R^{[2]}_{\Delta}$ as follows:
 $$
 R^{[1]}_{\Delta}(f_1,f_2)= (-1)^{n(m-l)} R^{[1]}(P_1,P_2),\quad  R^{[2]}_{\Delta}(f_1,f_2) = (-1)^{l(n-k)} R^{[2]}(P_1,P_2).
 $$

 The definitions of $\Delta$-resultants $R_{\Delta}^{[1]}$ and $R_{\Delta}^{[2]}$ are made in such a way that they coincide with the product resultants which are defined in the next section (see Theorem 3).

%
%

\section {Product formula in dimension one}
For Laurent polynomials $f_1, f_2$ as in (\ref{f1,f2}), let $\Pi_\Delta^{[1]}= \prod_{y_j\in Y} f_1^{m_{y_j}}(y_j)$ and $\Pi_\Delta^{[2]}=\prod_{x_i\in X} f_2^{m_{x_i}}(x_i)$ where $X=\{x_i\}$ and $Y=\{y_j\}$ are the sets of  non zero roots of $f_1$ and $f_2$ and $m_{y_j}, m_{x_i}$ are their multiplicities.

\begin {Theorem} If $a_ka_nb_lb_m\neq 0$ then the following identity holds:
\begin{equation} \label{prod1}
b_l^{-k} b_m^{n}\Pi_\Delta^{[1]}=(-1)^{kl+nm}a_k^{-l}a_n^{m}\Pi_\Delta^{[2]}.
\end{equation}
 \end {Theorem}
Let us recall an elementary proof of this classical theorem.
\begin{proof}
We have $f_1(z)=a_n(z-x_1)^{m_{x_1}}\cdot\cdots\cdot (z-x_{s})^{m_{x_s}}z^{k}$, so
$$
\Pi_\Delta^{[1]}= \prod_{y_j\in Y} \left[  a_n y_j^k \prod_{x_i\in X} (y_j-x_i)^{m_{x_i}} \right]^{m_{y_j}}=a_n^{m-l}\left[(-1)^{m-l} (b_l/b_m)\right]^k \Pi_\Delta^{[1,2]} ,
$$
where $\Pi_\Delta^{[1,2]}=\prod_{x_i\in X, y_j\in Y}(y_j-x_i)^{m_{x_i}m_{y_j}} $. Here we used the Vieta relation $\prod_{y_j\in Y}y_j^{m_{y_j}}= (-1)^{m-l}b_l/b_m$.
Thus we proved the identity
$
b_l^{-k} b_m^{n}\Pi_\Delta^{[1]}=a_n^{m-l}b_m^{n-k}\Pi_\Delta^{[1,2]} (-1)^{mk-lk}.
$

In a similar way
$
a_k^{-l}a_n^m\Pi_\Delta^{[2]}=a_n^{m-l}b_m^{n-k}\Pi_\Delta^{[2,1]}(-1)^{nl-kl},
$
 where $\Pi_\Delta^{[2,1]}=\prod_{x_i\in X, y_j\in Y}(x_i-y_j)^{m_{x_i}m_{y_j}} $.
But $\Pi_\Delta^{[1,2]}=\Pi_\Delta^{[2,1]} (-1)^{(m-l)(n-k)}$. Theorem 1 is proved.
\end{proof}

For Laurent polynomials $f_1,f_2$ as in (\ref{f1,f2}) let us define their product resultants $
R_{\Pi,\Delta}^{[1]}(f_1,f_2)$ and $R_{\Pi,\Delta}^{[2]}(f_1,f_2)$ by the formulas
\begin{equation} \label{prodres}
R_{\Pi,\Delta}^{[1]}=b_l^{-k} b_m^{n}\Pi_\Delta^{[1]}, \qquad R_{\Pi,\Delta}^{[2]}=a_k^{-l}a_n^{m}\Pi_\Delta^{[2]}.
\end{equation}

\begin{Theorem}
(1) The product resultants are polynomials in the coefficients of $f_1$ and $f_2$ (the expressions in (\ref{prodres}) themselves could have removable singularities at the hyperplanes where the extreme coefficients vanish);

(2) If the extreme coefficients are nonzero, i.e. $a_ka_nb_lb_m\neq 0$, the product resultants equal to zero exactly on the pairs of Laurent polynomials having  common root in $\C^*$;

(3) The product resultants have degrees  $(m-l)$ and $(n-k)$ in the coefficients of $f_1$ and $f_2$ correspondingly;

(4) The coefficient of the monomial $a_k^{m-l}b_m^{n-k}$ in $R_{\Pi,\Delta}^{[1]}$ and $R_{\Pi,\Delta}^{[2]}$ is equal to $(-1)^{k(m-l)}$ and $(-1)^{l(n-k)}$ respectively.
\end{Theorem}

\begin{proof}
The expression $b_l^{-k} b_m^{n}\Pi_\Delta^{[1]}$  obviously is a polynomial of degree $m-l$ in coefficients of $f_1$ and the expression $a_k^{-l}a_n^{m}\Pi_\Delta^{[1]}$ is obviously a polynomial in coefficients of $f_2$ of degree $n-k$. Since two expressions are equal up to sign we have proved (1) and (3).

It is clear that $R_{\Pi,\Delta}^{[1]}$ vanishes if and only if $f_1(z)=0$ for some root $z$ of $f_2$, so $z$ is a common root. The same is true for $ R_{\Pi,\Delta}^{[2]}$.

For part (4) let us note that the values of the monomial $a_k^{m-l}b_m^{n-k}$  in $b_l^{-k} b_m^{n}\Pi_\Delta^{[1]}$ come from multiplying the term $a_kz^k$ over roots of $f_2$. Using the Vieta formula we have:
$
b_l^{-k} b_m^{n}\prod a_k z^k =(-1)^{k(m-l)}a_k^{m-l}b_m^{n-k}.
$
\end{proof}

\begin{Theorem}
Assume that $k=l=0$. Then the product resultants  coincide with the Sylvester resultants:
$$
R_{\Delta}^{[1]}=R_{\Pi,\Delta}^{[1]} = (-1)^{kl+nm} R_{\Pi,\Delta}^{[2]}= (-1)^{kl+nm}R_{\Delta}^{[2]}.
$$
\end{Theorem}

\begin{proof}
Both functions $R_{\Delta}^{[1]}$ and $R_{\Pi,\Delta}^{[1]}$ are polynomials in the coefficients of $f_1,f_2$ of the same degree. They both vanish on the set of pairs of polynomials having a common root in $\C^*$. Since the set of pairs of polynomials having a common root is irreducible the polynomials $R_{\Delta}^{[1]}$ and $R_{\Pi,\Delta}^{[1]}$ are proportional.  They have the same coefficient in  front of the monomial $a_k^{m-l}b_m^{n-k}$, so they are equal. A similar argument works for $R_{\Delta}^{[2]}$ and $R_{\Pi,\Delta}^{[2]}$.
\end{proof}

Keeping in mind multidimensional generalizations in the section 3 we will present  another topological proof of Theorem 1 and in the section 4 we will present an  algorithm for computing  the product resultant which does not rely on the Sylvester determinant.

\section { Weil reciprocity law}

\subsection {Weil symbol and Weil law}

 Let  $f$ and $g$ be two meromorphic functions on a compact Riemann surface  $S$. About  each point $p\in S$ one can choose a local parameter $u$ such that  $u(p)=0$ and consider the  Laurent expansions $f=c_1u^{k_1}+\ldots$, $g=c_2u^{k_2}+\ldots$ of $f$ and $g$, with dots are standing for higher order terms. One can check that the expression
 $$
 \{f,g\}_p=(-1)^{k_1k_2} c_1^{-k_2}c_2^{k_1}
 $$
 is independent of the choice of $u$. The number $\{f,g\}_p$ is called the {\it Weil symbol} of $f$ and $g$ at the point $p$.

\begin{ex}
If $p$ is a  zero of $g$ of multiplicity $m_p$ and $f(p)\ne 0,\infty$,  then $\{f,g\}_p=f^{-m_p}(p)$.
If $p$ is a  zero of $f$ of multiplicity $m_p$ and $g(p)\ne 0,\infty$,  then $\{f,g\}_p=g(p)^{m_p}$.

\end{ex}

\begin{ex}
Consider  $f_1$, $f_2$ from (\ref{f1,f2}) as the functions on  $\C P^1$. The main terms of Laurent expansions of $f_1,f_2$ at 0 are $a_kz^k, b_lz^l$ respectively, so $\{f_1,f_2\}_0=(-1)^{kl}a_k^{-l}b_l^k$.

Let $w=1/z$ be the local parameter on $\C P^1$ at $\infty$. Then the main terms of Laurent expansions of $f_1,f_2$ at $\infty$ are $a_nw^{-n}, b_mw^{-m}$ respectively, so $\{f_1,f_2\}_\infty=(-1)^{nm}a_n^mb_m^{-n}$.
\end{ex}

Let $D\subset S$ be a finite set  containing all points where  $f$ or $g$ is equal to 0 or to $\infty$ (we assume that each function $f$,$g$ is not identically equal to zero at each connected component of $S$).

\begin{Theorem}[Weil reciprocity law] For any couple of meromorphic functions $f$ and $g$ the following relation holds: \begin{equation}\label{W}
\prod_{p\in D} \{f,g\}_p =1.
\end{equation}
\end{Theorem}

A compact Riemann surface $S$ equipped with its field of meromorphic functions can be considered as an algebraic curve equipped with its field of rational functions. Under such consideration Theorem 4 becomes purely algebraic.

\begin{Corollary} Consider  $f_1$, $f_2$ from (\ref{f1,f2}) as the functions on  $\C P^1$. Let $X,Y$ be sets of non zero roots of $f_1,f_2$. Assume that $X\cap Y=\emptyset$ and roots $x_i\in X, y_j\in Y$ have multiplicities $m_{x_i}$, $m_{y_j}$. Then according to the Weil reciprocity law and examples 1 and 2

$$b_l^{-k} b_m^{n}\prod_{y_j\in Y} f_1^{m_{y_j}}(y_j)=(-1)^{kl+nm}a_k^{-l}a_n^{m}\prod_{x_i\in X} f_2^{m_{x_i}}(x_i).$$

\end{Corollary}

Thus  Theorem 1 could be considered as a corollary of the Weil reciprocity law. On the other hand Theorem 1  provides an elementary proof of the Weil reciprocity law in the case under consideration. In the general case the Weil reciprocity law  also can be reduced using Newton polygons to similar elementary arguments \cite{Kh2}.

\subsection {Topological extension of the Weil reciprocity law.} Let $S$ be a  Riemann surface (not necessary compact) and let $D\subset S$ be a discrete subset. The  {\it Leray coboundary operator} $\delta$ associates to every point $p\in D$ an element $\delta(p)\in H_1(S\setminus D,\Z) $  represented by a small circle centered at $p$ with the counterclockwise orientation.   Let $M$ be the multiplicative group of meromorphic functions on $S$  which are regular and nonzero on $S\setminus D$.

\begin{Theorem} To each couple $f,g\in M$ one can associate a map $\{f,g\}: H^1(S\setminus D,\Z)\rightarrow \C^*$ such that the following properties hold: 1) for each $p\in D$ the image $\{f,g\}(\delta(p))$ of the cycle $\delta(p)$ under the map $\{f,g\}$   is equal to the Weil symbol $\{f,g\}_p$; (2) $\{f,g\}= \{g, f\}^{-1}$;
(3) for any triple $f,g,\phi \in M$ the identity $\{f,g\phi\}=\{f,g\}\{f,\phi\}$ holds.
\end{Theorem}

A simple proof of Theorem 5 can be found in \cite{Kh4}. If the surface $S$ is compact then the following relation between the cycles $\delta(p)$ holds:

\begin{Lemma}
The element $\sum_{p\in D} \delta(p)\in H_1( S\setminus D, \Z)$ is equal to zero.
\end{Lemma}

\begin{proof}Indeed the cycle $-\sum_{p\in D} \delta(p)$ is the boundary of $S\setminus \bigcup_{p\in D}B_{p}$ where $B_{p}$ is the open ball centered in $p$ with the boundary $\delta (p)$.
\end{proof}

The Weil reciprocity law follows from Theorem 5 and Lemma 1: $\prod_{p\in D} \{f,g\}_p=1$ because in $H_1(S\setminus D)$ the identity $\sum_{p\in D} \delta (p)=0$ holds.

Let us reformulate Lemma 1 in the case related to the torus $\C^*= \C P^1\setminus \{0,\infty\}$. We will work with $S= \C P^1$ and $D=\{0,\infty\}\cup D'$ where $D'$ is a finite set containing the sets $X,Y$  of non zero roots of the functions $f_1,f_2$ from  (\ref{f1,f2}),  i.e. containing non zero roots of  $P=f_1f_2$. Let $\Delta$ be the Newton polyhedron of
$P$. Then $\Delta$ is the segment with vertices $A_0 =k+l$ and $A_\infty = n+m$.  Let $T^1\subset \C^*$ be the circle  $|z|=1$  orientated by the form $d(\arg z)$. Let  $T^1_{A_0}, T^1_{A_\infty}$ be the cycles in $\C^*\setminus D'$ given by $\frac{1}{\lambda}T^1, \lambda T^1$  where $|\lambda|$ is  big enough.  Let $k_{A_0}=1$ and $k_{A_\infty}=-1$.
\begin{Theorem}[one dimensional topological theorem]
In the notations above the identity
$\sum_{p\in D'}\delta(p)= -(k_{A_0}T^1_{A_0}+ k_{A_\infty}T^1_{A_\infty})$ holds.
\end{Theorem}
\begin{proof} Theorem 6 immediately follows from Lemma 1 because  $T^1_{A_0}=\delta(0)$ and  $T^1_{A_\infty}=-\delta(\infty)$.
\end{proof}

The identity $\prod_{p\in D} \{f_1,f_2\}^{-1}_p = \{f_1,f_2\}_0\{f_1,f_2\}_\infty$ follows from Theorems 5 and 6. It can be rewritten as
$\Pi_\Delta^{[1]}/\Pi_\Delta^{[2]}=(-1)^{-kl+nm} a_k^{-l}b_l^{k}a_n^mb_m^{-n}.$
Thus we obtained a topological proof of Theorem 1.

\section{Sums over roots of Laurent polynomial and elimination theory}

\subsection{Sums over roots of Laurent polynomial}

Let $z\in \C^*$ be a root of  multiplicity $\mu(z)$ of the Laurent polynomial $P$. Then for any Laurent polynomial $f$ the following theorem holds.

\begin{Theorem}
The sum $\sum f(z)\mu(z)$  over all roots $z\in \C^*$ of $P$ is equal to $-(res_0\omega+res_\infty \omega)$, where
$
\omega=f\frac{dP}{P}= \frac{\partial P}{\partial z}\cdot\frac{fz}{P}\cdot\frac{dz}{z}.
$
\end{Theorem}

\begin{proof} Theorem 7 follows from the Cauchy residue formula since the residue of the form $\omega$ at a root $z$ is equal to $f(z)\mu(z)$.
\end{proof}

Theorem 7 provides an explicit formula for the sum $\sum f(z)\mu(z)$: in contrast to the roots of $P$, the points $0,\infty$ are independent of the coefficients of $P$ and therefore the residues of $\omega$ at $0$ and $\infty$ could be explicitly computed.

\subsection{Elimination theory related to the one-dimensional case}
Let $P$  and  $f$ be  Laurent polynomials as above. Here we explain how to find any symmetric function of the set of values $\{f(z)\}$ over all roots $z\in \C^*$ (each root $z$ is taken with multiplicity $\mu(z)$).

Denote by $f^{(k)}$  the number $ f^{(k)}=\sum_z f^k(z)\mu(z).$
By Theorem 7 one can  calculate $f^{(k)}$ for any $k$ explicitly. The power sum symmetric polynomials form a generating set for the ring of symmetric polynomials.

\begin{Corollary} One can find explicitly all symmetric functions of $\{f(z)\}$, construct a monic polynomial whose roots are $\{f(z)\}$, and eliminate $z$ from the definition of the set $\{f(z)\}$.
In particular, one can compute the products $\Pi_\Delta^{[1]}, \Pi_\Delta^{[2]}$ defined in the section 3.
\end{Corollary}

\section{Developed systems and combinatorial coefficients}

Let $\Delta=(\Delta _1,\dots,\Delta _n)$ be an $n$-tuple of convex polyhedra in  $\Bbb R^n$,
and let $\sum\Delta_i= \Delta _1+\dots+\Delta _n$ be their Minkowski sum. Each face $\Gamma $ of
the polyhedron $\sum\Delta_i$ can be uniquely represented as a sum
$
\Gamma =\Gamma _1+\dots+\Gamma _n,
$
where $\Gamma_i$ is a face of $\Delta_i$.

An $n$-tuple $\Delta$  is called
{\it developed} if for each face $\Gamma$ of the polyhedron
$\sum\Delta_i$, at least one of the terms $\Gamma_i$ in its decomposition is a vertex.

The system of equations $f_1=\ldots=f_n=0$ on $(\C^*)^n$, where $f_1,\ldots,f_n$ are Laurent polynomials, is called developed  if the $n$-tuple  $(\Delta_1,\dots,\Delta_n)$ of their Newton polyhedra is developed.

For a developed $n$-tuple of polyhedra $\Delta$,
a map $h:\partial \sum \Delta_i\rightarrow  \partial \Bbb R^n_+ $ of the
boundary $\partial \sum \Delta_i$ of $\sum \Delta_i$ into the boundary of the
positive octant is called {\it characteristic} if the component
$h_i$ of the map $h=(h_1,\dots,h_n)$ vanishes precisely on the
faces $\Gamma$, for which the $i$-th term $\Gamma_i$ in the
decomposition is a point (a vertex of the polyhedron $\Delta _i$). One can show that   the space of characteristic maps is nonempty and connected.
The preimage of the origin under a characteristic map is
precisely the set of all vertices of the polyhedron $\sum\Delta_i$.

The {\it combinatorial coefficient} $k_A$ of a vertex $A$ of $\sum\Delta_i$ is the local degree of the germ
$$
h:(\partial\sum\Delta_i,A)\to (\partial \Bbb R^n_+,0)
$$
of a characteristic map restricted to the boundary $\partial\sum \Delta_i$ of $\sum \Delta_i$. The combinatorial coefficient is independent of a choice of a characteristic map, but it depends on the choice of  the orientation of $\sum\Delta_i$ and $\Bbb R^n_+$. The first one is given by the orientation of the space of characters on $(\C^*)^n$. The second is defined by an ordering of the polyhedra   $\Delta_1,\dots, \Delta_n$ in $n$-tuple $\Delta$. Both orientations are an arbitrary choice, and after changing each of them the combinatorial coefficient will change sign. For more detailed discussion of combinatorial coefficients see \cite{GKh1},\cite{Sop},\cite{Kh2}.

Let us discuss combinatorial coefficients in the two-dimensional case. Two polygons $\Delta_1, \Delta_2 \subset \R^2$ are developed if and only if they do not have parallel sides with the same direction of the outer normals (see figure \ref{comb}). If $\Delta_1,
 \Delta_2$ are developed  each side of $\Delta_1+\Delta_2$ comes either from $\Delta_1$ or from $\Delta_2$. That is each side of $\Delta_1+\Delta_2$ is either the sum of a side of $\Delta_1$ and a vertex of $\Delta_2$, or the sum of a vertex of $\Delta_1$ and a side of $\Delta_2$.

The two types of sides are labeled by 2 (dashed in the picture) and 1 (solid in the picture) respectively. Giving $\R^2$ and the positive octant the standard orientations we can find the local degree of a characteristic map at a vertex. It is equal to 0 if neighbouring edges have the same label, to $+1$ if the label at $A$ is changing from 2 to 1 in the counter clockwise direction, and to $-1$ if it is changing from 1 to 2.

So the combinatorial coefficient $k_A$ of a vertex $A\in \Delta$ is equal to 0 if neighbouring edges have the same label, and $+1$ or $-1$ (depending on the orientation) if the labeling changes at $A$. The only possible value of combinatorial coefficient in dimension 2 is $-1, 0$ or $+1$ because the local mapping degree of one dimensional manifolds could take only these values. The combinatorial coefficient in dimension $\geq 3$ could be any integer number.

\begin{figure}
\begin{center}
\includegraphics*[scale=0.15]{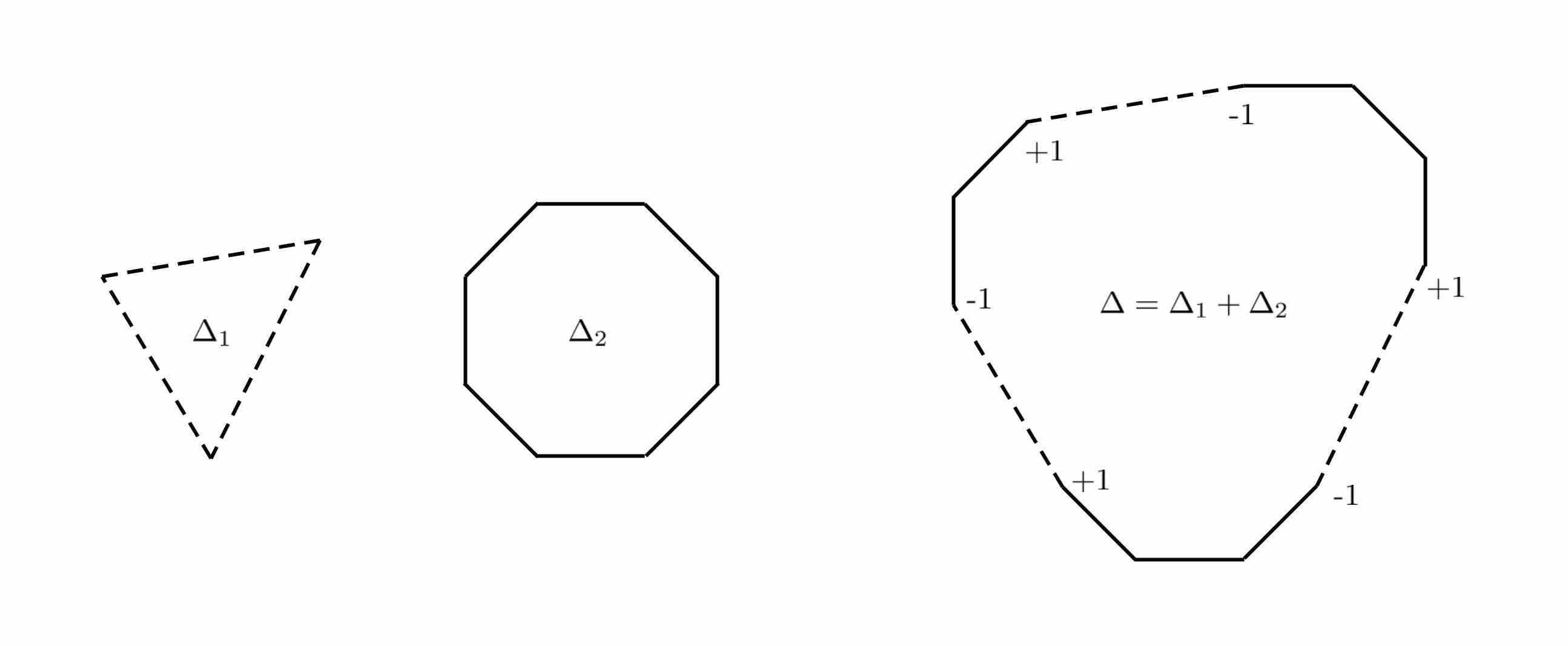}
\caption{Combinatorial coefficient in dimension 2.}\label{comb}
\end{center}
\end{figure}

\section{Topological theorem.}

\subsection{Grothendieck cycle.}
Let $z$ be an isolated root of a system $f_1=\ldots=f_n=0$ on $(\C^*)^n$ where $f_1,\ldots,f_n$ are Laurent polynomials. The Grothendieck cycle $\gamma_z$ is a class in the group of $n$-dimensional homologies of the complement $(U\setminus\Gamma)$ of a small neighborhood $U$ of the point $z$, of the hyperplane $\Gamma$ defined by the equation $P=f_1\ldots f_n=0$. For almost all small enough $\epsilon=(\epsilon_1,\ldots\epsilon_n) \in \R_+^n$, the subset $\gamma_{z,\epsilon}$ defined by $|f_i|=\epsilon_i$ is a smooth compact real submanifold of $(U\setminus \Gamma)$. The Grothendieck cycle $\gamma_z$ is the cycle of the submanifold $\gamma_{z,\epsilon}$ for small enough $\epsilon$ oriented by the form $d(argf_1)\wedge\ldots\wedge d(argf_n)$. The orientation of the Grothendieck cycle depends on the order of the equations $f_1=0,\ldots,f_n=0$.

\subsection{The cycle related to a vertex of the Newton polyhedron.}

Let $\Gamma\subset (\C^*)^n$ be a hypersurface $P=0$, where $P$ is a Laurent polynomial with Newton polyhedron $\Delta(P)$. Let $T^n\subset (\C^*)^n$ be the torus  $|z_1|=\ldots=|z_n|=1$  orientated by $\omega=d(\arg z)\wedge\ldots\wedge d(\arg(z_n)$. The sign of $\omega$ depends on the order of variables, so the sign of a cycle $T^n$ depends on the orientation of the space $\R^n$ of characters on $(\C^*)^n$.

For every vertex $A$ of $\Delta(P)$ we will assign an $n$-dimensional cycle $T_A^n$ in  $ (\C^*)^n \setminus \Gamma$ defined up to homological equivalence. For this denote by $\xi_A=(\xi_1,\ldots,\xi_n)$  an integer covector  such that the inner product of $x\in \sum\Delta_1$ with $\xi_A$ attains its maximum value at  $A$. Consider the 1-parameter subgroup $\lambda(t)=(t^{\xi_1},\ldots,t^{\xi_n})$ of $(\C^*)^n$. For $t$ with large enough absolute value $|t|$ the translation $\lambda(t)T^n$ of $T^n$ by the subgroup $\lambda(t)$ does not intersect  the hypersurface $\Gamma$.

Let us define $T^n_A$ as a cycle $\lambda(t)T^n$ with $|t|$ large enough. The definition makes sense  because the homology class of $\lambda(t)T^n$ in $ (\C^*)^n \setminus \Gamma$  does not depend on the choice of $\xi_A$ and $t$, provided $|t|$ is large enough.

\subsection{Topological theorem for $n$ Laurent polynomials.}

Let  $f_1,\ldots,f_n$ be Laurent polynomials with developed Newton polyhedra $\Delta_1,\ldots,\Delta_n$. Let $\Gamma$ be a hypersurface in $(\C^*)^n$ defined by the equation  $P= f_1\ldots f_n=0$. In \cite{GKh},\cite{GKh1} the following theorem is proved.

\begin{Theorem}
In  $ (\C^*)^n \setminus \Gamma$  the sum of the Grothendieck cycles $\gamma_z$ over  all roots $z$ of the system $f_1=\ldots=f_n=0$ is homologous to the cycle $(-1)^n\sum k_AT^n_A$, where the sum is taken over all vertices  $A$ of $\sum \Delta_i$ and $k_A$ is the combinatorial coefficient  at the vertex $A$.
\end{Theorem}

The signs in the topological theorem depend on the choice of order of variables $z_1,\ldots, z_n$ and on the choice of order of functions  $f_1,\ldots,f_n$. In the statement these orders are fixed in an arbitrary way. Changing the order of the variables changes the sign of cycle at vertices $T^n_A$ and all of the combinatorial coefficients.  Choosing a different order for the equations will change the signs of all the Grothendieck cycles $\gamma_z$, and all of the combinatorial coefficients as well.

\subsection{Topological theorem for $(n+1)$ Laurent polynomials.}
We will say that the collection of polyhedra $\Delta_1,\ldots,\Delta_{n+1}$ is {\it $i$-developed } if the collection  with $\Delta_i$ removed is developed. In this section we present a version of the topological theorem applicable for a collection of $n+1$ Laurent polynomials which is $i$-developed and $j$-developed for some $1\leq i<j\leq n$

 Let $i,j$ be indexes such that $1\leq i< j\leq n+1$. We will associate with $i,j$ the permutation $\{k_1,\dots,k_{n+1}\}$  of $\{1,\dots,n+1\}$ defined by the following relations $k_1=i, k_2=j$, and $k_3<\dots<k_{n+1}$.

The following lemma is obvious.

\begin{Lemma}
A collection of $(n+1)$ polyhedra $\Delta_1,\ldots,\Delta_{n+1}$ in $\R^n$ is $i$-developed and $j$-developed for some $i<j$ if and only if the collection $(\Delta_{i}+\Delta_{j}), \Delta_{k_3},\dots, \Delta_{k_{n+1}}$ is developed.
\end{Lemma}

Let $\Delta_1,\ldots,\Delta_{n+1}$ be an $i$-developed and $j$-developed collection of Newton polyhedra. Let us denote by $k^{i,j}_A$ the combinatorial coefficient of a vertex $A\in \sum\Delta_i$ associated with  the  the collection $\Delta_i+\Delta_j,\Delta_{k_3},\ldots,\Delta_{k_{n+1}}$.
Let $f_1,\dots,f_{n+1}$ be  Laurent polynomials with Newton polyhedra $\Delta_1,\dots,\Delta_{n+1}$ such that the system
\begin{equation}\label{all f} f_1=\dots=f_{n+1}=0
\end{equation} is not consistent in $(\C^*)^n$.
Denote by $X_i$ the set of all roots $x$ of the system
\begin {equation}\label{X_i}
f_1=\dots=\hat f_i=\dots=f_{n+1}=0,
\end{equation}
where the equation $f_i=0$ is removed. Denote by  $X_j$ the set of all roots $y$ of the system
\begin{equation}\label{X_j}
f_1=\dots=\hat f_j=\dots=f_{n+1}=0,
\end{equation}
where the equation $f_j=0$ is removed.

\begin{Theorem}
 Assume that the Newton polyhedra $\Delta_1,\dots,\Delta_{n+1}$ of the Laurent polynomials $f_1,\dots,f_{n+1}$ are $i$-developed and $j$-developed and that the system (\ref {all f}) is not consistent in $(\C^*)^n$. Then in the group $H_n((\C^*)^n \setminus \Gamma,\Bbb Z)$ the identity
$$
(-1)^{j-2}\sum _{x\in X_i}\gamma_x+ (-1)^{i-1}\sum _{y\in X_j}\gamma_y=(-1)^n\sum k_A^{i,j}T^n_A
$$
holds, where $\gamma_x$ and $\gamma_y$ are the Grothendieck cycles of the roots $x$ and $y$ of the systems (\ref{X_i}), (\ref{X_j}) and the summation on the right is taken over all vertices  $A$ of $\Delta=\Delta_1+\dots+\Delta_n$.
\end{Theorem}

\begin{proof}
According to the topological theorem the sum of the Grothendieck cycles $\gamma_z$ over the set $X_{i,j}$ of all roots $z$ of the system
\begin{equation}\label{i,j}
f_if_j=f_{k_1}=\dots=f_{k_{n+1}}=0
\end{equation}
is equal to $(-1)^n k_A^{i,j}T_A$. The set $X_{i,j}$ is equal to $X_i\cup X_j$ where $X_i$ is the set of roots $x$ of the system  $f_j=f_{k_1}=\dots=f_{k+1}=0$  and $X_j$ is the set of roots $y$ of the system $f_i=f_{k_1}=\dots=f_{k+1}=0$. If $z=x\in X_i$ then the cycle $\gamma_z$ is equal to the cycle $(-1)^{j-2}\gamma_x$ for the system (\ref{X_i}). The sign $(-1)^{j-2}$ in the identity appears because of the change of the order equations from $f_j=f_{k_1}=\dots=f_{k_{n+1}}=0$ to $f_1=\dots=\hat f_i=\dots=f_{n+1}=0$. In a similar way if $z=y\in X_j$ then the cycle $\gamma_z$ is equal to the cycle $(-1)^{i-1}\gamma_y$ for the system (\ref{X_j}).
\end{proof}

\section {Parshin reciprocity laws}

\subsection{Analog of the determinant of $n+1$ vectors in $n$-dimesional space over  $\Bbb F_2$.}

A determinant of $n$ vectors in $n$-dimensional space $L_n$ over the field $\Bbb F_2=\Z/2\Z$ is unique non-zero multilinear function on $n$-tuples of vectors in $L_n$ which is invariant under the $GL(n,\Bbb F_2)$ action and which has value 0 if the $n$-tuple is dependent.

It turns out that there exists unique function on $(n+1)$-tuples of vectors in $L_n$ having exactly the same properties (see \cite{Kh2}, \cite{Kh5}).

\begin{Theorem}
There exists a unique non zero function  $D$ on $(n+1)$-tuples of vectors in $L_n$ satisfying the following properties:

(i) $D$  is $GL(n,\Bbb F_2)$ invariant, i.e. for any $A\in GL( n,\Bbb F_2)$ the equality $D(k_1,\ldots,k_{n+1})=D(A(k_1),\ldots,A(k_{n+1}))$ holds;

(ii) if the rank of $k_1,\ldots,k_{n+1}$ is $<n$ then $D(k_1,\ldots,k_{n+1})=0$;

(iii) $D$ is multilinear.
\end{Theorem}

Let us present two explicit formulas for $D$:

I) If the rank of  $k_1,\ldots,k_{n+1}$ is $ <n$ then $D(k_1,\dots,k_{n+1})=0$ (see (ii)). If the rank is $n$ then there is a unique non-zero collection $\lambda_1,\dots, \lambda_{n+1}\in \Bbb F_2$ such that $\lambda_1 k_1+\dots +\lambda_{n+1}k_{n+1}=0$ and . In this case  $D(k_1,\dots,k_{n+1})={1+\lambda_1+\dots+\lambda_{n+1}}$.

II) If on the space $L_n$ the coordinates are fixed, then
$ D(k_1,\ldots,k_{n+1})= \sum_{j>i}\Delta_{ij},$
where $\Delta_{ij}$ is the determinant of $(n\times n)$ matrix whose first $n-1$ columns are coordinates of vectors $k_1,\ldots,k_{n+1}$ with vectors $k_i, k_j$ removed and the last column is the coordinatewise product of vectors $k_i$ and $k_j$.

For a vector $v\in \Z^n$ let $\tilde v\in \Bbb F_2^n$ be its mod 2 reduction. For an $(n+1)$-tuple of vectors $v_1\ldots v_{n+1} \in \Z^n$ we define $D(v_1\ldots v_{n+1})$ as $D(\tilde v_1\ldots \tilde v_{n+1})$.

The determinant of a matrix $A$ over $\R$ is the volume of the oriented
parallelepiped spanned by the columns of A. It turns out that the function $D$ also computes the volume of some figure  (see \cite{Kh3}). This property of $D$ allows to fit it into the topological version of Parshin reciprocity laws (see \cite{Kh5} and the section 8.3).

\subsection{ Parshin symbols of monomials}

Consider $n+1$ monomials $c_1\mathbf z^{\mathbf k_1},\dots, c_{n+1}\mathbf z^{\mathbf k_{n+1}}$  with nonzero coefficients $c_i\in \Bbb C^*$ in $n$ complex variables $\mathbf z =(z_1,\dots, z_n)$, $\mathbf k_i\in(\Bbb Z)^n, \mathbf k_i=(k_{i,1},\dots,k_{i,n})$, $c_i\mathbf z^{\mathbf k_i}=c_i z_1^ {k_{i,1}}\cdot\dots \cdot z_n^ {k_{i,n}}$. The {\it Parshin Symbol} $[c_1\mathbf z^{\mathbf k_1},\dots, c_{n+1}\mathbf z^{\mathbf k_{n+1}}]$ of the sequence $c_1\mathbf z^{\mathbf k_1}, \dots, c_{n+1}\mathbf z^{\mathbf k_{n+1}}$ is equal by definition to
$$
(-1)^{D(\mathbf k_1,\dots,\mathbf k_{n+1})}c_1^{-\det(\mathbf
k_2,\dots, \mathbf k_{n+1})} \dots c_{n+1}^{(-1)^{n+1}\det(\mathbf
k_1,\dots, \mathbf k_{n})}=
$$
$$
 = (-1)^{D(\mathbf k_1,\dots,\mathbf
k_{n+1})}\exp \left(-\det \left(\begin{array}{cccc}\ln c_1&k_{1,1}&\dots&k_{1,n+1}\\
\vdots&\vdots&&\vdots\\
\ln c_{n+1}&k_{n+1,1}&\dots&k_{n+1,n+1}\end{array}\right)\right),
$$
where $D:(\Bbb Z^n)^{n+1} \rightarrow \Bbb Z/2\Bbb Z$ is the function defined in the previous section.
\begin{ex} The Parshin symbol $[c_1z^{k_1}, c_2z^{k_2}]$ of
functions  $c_1z^{k_1}, c_2z^{k_2}$ in one variable $z$
is equal to $(-1)^{k_1k_2}c_1^{-k_2}c_1^{k_1}$,
thus it is equal to the Weil symbol $\{c_1z^{k_1}, c_2z^{k_2}\}_0$ of these functions at
the origin $z=0$.
\end{ex}
By definition, the Parshin symbol is skew-symmetric, so for example,
$$[c_1\mathbf z^{\mathbf k_1}, c_2\mathbf z^{\mathbf k_2}, \dots,
c_{n+1}\mathbf z^{\mathbf k_{n+1}}] = [c_2\mathbf z^{\mathbf k_2},
c_1\mathbf z^{\mathbf k_1}, \dots, c_{n+1}\mathbf z^{\mathbf
k_{n+1}}]^{-1},$$
and multiplicative, so for example, if
$c_1\mathbf z^{\mathbf k_1}=a_1 b_1\mathbf z^{\mathbf l_1+\mathbf m_1}$,
then
$$
[c_1\mathbf z^{\mathbf k_1},\dots, c_{n+1}\mathbf z^{\mathbf k_{n+1}}] =[a_1\mathbf z^{\mathbf
l_1}, \dots, c_{n+1}\mathbf z^{\mathbf l_{n+1}}][b_1\mathbf z^{\mathbf
m_1},\dots, c_{n+1}\mathbf z^{\mathbf m_{n+1}}].
$$

\subsection{Topological version of Parshin laws.}

 The Parshin  reciprocity laws (see \cite{Par}, \cite{FPar}) are applicable  to $(n+1)$ rational functions on an $n$-dimensional  algebraic variety over an algebraically closed field of any characteristic. They contain several general relations between the Parshin symbols of these $(n+1)$ functions analogous to the relation between the Weil symbols of two functions given in the Weil reciprocity law for an algebraic curve. We will need a topological version of Parshin's laws over $\Bbb C$ in the special situation that the algebraic variety is $(\C^*)^n$ and the $(n+1)$ functions are Laurent polynomials $f_1,\dots,f_{n+1}$. Let us state needed facts for that special situation (for general case see \cite{Kh5}).

Let $\Gamma$ be a hypersurface in $(\C^*)^n$ defined by the equation  $P= f_1\ldots f_n=0$.
According to the topological version of the Parshin reciprocity laws there is a  map $[f_1,\dots,f_{n+1}] : H^n((\C^*)^n\setminus \Gamma,\Z)\rightarrow \C^*$ having the following properties:

1)  The map $ [f_1,\dots,f_n]$ depends skew symmetrically on the components $f_i$, so for example $ [f_1,f_2,\dots,f_{n+1}]= [f_2,f_1,\dots,f_{n+1}]^{-1}$.

2) Let $A=A_1+\ldots+A_{n+1}$ be a vertex of $\Delta_1+\dots+\Delta_{n+1}$
where $A_i$ is a vertex in $\Delta_i$. Let $c_i\mathbf z^{\mathbf k_i}$ be  the  monomial with the coefficient $c_i$ in $f_i$ corresponding to  $A_i\in \Delta_i$. Then $[f_1,\ldots,f_{n+1}](T^n_A)=[c_1\mathbf z^{\mathbf k_1},\ldots, c_{n+1}\mathbf z^{\mathbf k_{n+1}}]$ where $T^n_A$ is the cycle corresponding to the vertex $A$.

3) Let $z$ be a root of multiplicity $\mu(z)$ of the system $f_1=\ldots=\hat f_i=\dots=f_{n+1}=0$, where the equation $f_i=0$ has been removed.  Let $\gamma_z$ be the corresponding Grothendieck cycle. Assume that $f_i(z)\neq 0$.  Then $[f_1,f_2,\ldots,f_{n+1}](\gamma_z)=f_i(z)^{(-1)^i\mu(z)}$.

\section {Product over roots of a system of equations}

Let $\Delta=(\Delta_1,\ldots,\Delta_{n+1})$ be $(n+1)$ Newton polyhedra in the lattice  $\Z^n$ and let $\Omega_{\Delta}$ be the  space  of $(n+1)$-tuples of Laurent polynomials $(f_1,\ldots f_{n+1})$ such that  the Newton polyhedron of $f_i$ is contained in $\Delta_i$. We will define a rational function $\Pi^{[i]}_{\Delta}$ on the space $\Omega_{\Delta}$ which we will call {\it the product of $f_i$ over the common zeros  of  $f_j$ for $j\neq i$}. Let $U^i_{\Delta}\subset \Omega_{\Delta}$ be the Zariski open set defined by the following condition: $(f_1,\ldots,f_{n+1})\in U^i_{\Delta}$ if and only if the set $Y^i_{\Delta}\subset (\C^*)^n$ of  common zeros  of  $f_j$ for $j\neq i$ is finite and the number of points in $Y^i_{\Delta}$ (counting with multiplicities) is equal to $n!Vol (\Delta_1,\ldots, \hat \Delta_i, \ldots,\Delta_{n+1})$ (the polyhedron $\Delta_i$ is omitted in this mixed volume).

\begin{Definition}  We  define the function $\Pi_{\Delta}^{[i]}$ on $U^i_{\Delta}$ as follows. If the set $Y^i_{\Delta}\subset (\C^*)^n$ is empty (i.e. if $n!Vol (\Delta_1,\ldots , \hat \Delta_i, \ldots,\Delta_{n+1})= 0$) then $\Pi^{[i]}_{\Delta} \equiv 1$. Otherwise

$$
\Pi_{\Delta}^{[i]}(f_1,\ldots,f_{n+1})=\prod_{x\in Y^i_{\Delta}} f_i^{m_x}(x),
$$
where $m_x$ is the multiplicity of the common zero $x\in Y^i_{\Delta}$ of the Laurent polynomials $f_j$ for $j\neq i$.

\end{Definition}

\begin{Lemma}
The function $\Pi^{[i]}_{\Delta}$ is regular on $U^{i}_{\Delta}$. It can be extended to a rational function on  $\Omega_{\Delta}$.
\end{Lemma}
\begin{proof} Let $\tilde U^i_{\Delta}\subset U^i_{\Delta} $ be the Zariski open set in which common zeros of $f_j$ with $j\neq i$ have multiplicity one. The function $\Pi^{[i]}_{\Delta}$ is obviously regular on $\tilde U^i_{\Delta}$. By the removable singularity theorem it is regular in  $U^i_{\Delta} $.  By definition $\Pi^{[i]}_{\Delta}$ is an algebraic  single-valued function on $U^{i}_{\Delta}$.  Thus, it is rational function on $\Omega_{\Delta}$.
\end{proof}

Note that even if $(f_1,\ldots,f_{n+1})\notin U^i_{\Delta}$ then the product of $f_i$ over the common roots of $f_j$ for $j\neq i$ is well defined if the set of common roots is finite. But this product is not  necessarily equal to $\Pi^{[i]}_{\Delta}(f_1,\ldots,f_{n+1})$.

 Assume that the collection  $\Delta=(\Delta_1,\ldots,\Delta_{n+1})$ is $i$-de\-ve\-loped. Consider the space $\Omega_\Delta$ of $(n+1)$-tuples $(f_1,\dots,f_{n+1})$ of Laurent polynomials whose Newton polyhedra are contained correspondingly in  $(\Delta_1,\ldots,\Delta_{n+1})$. Denote by $\Omega_\Delta^{i}$ the open subset in $\Omega_\Delta$ defined by the condition that the Newton polyhedron of $f_j$ is $\Delta_j$ for $j\neq i$.

\begin{Theorem}
 The function
$\Pi_\Delta^{[i]}$ is regular  on $\Omega_\Delta^{i}$. Moreover, there exists a monomial $M_i$ in vertex coefficients of all the $f_j$ for $j\neq i$ such that the product $M_i\Pi_\Delta^{[i]}$ is a polynomial.
\end{Theorem}

\begin{proof}
Since the system $f_1=  \ldots=\hat f_i=\dots =f_{n+1}=0$ is developed, the number of roots counting with multiplicities is constant on $\Omega_\Delta^{i}$ and $\Omega_\Delta^{i}\subset U^{i}_{\Delta}$. Therefore the function $\Pi_\Delta^{[i]}$ is regular on $\Omega_\Delta^{i}$. Thus  there exists a monomial $M_i$ in the vertex coefficients of $f_j$ for $j\neq i$ such that the product $M_i\Pi_\Delta^{[i]}$ is a polynomial on $\Omega_\Delta$.
\end{proof}

In Section 13.4 we will present an algorithm for computing the function $\Pi_\Delta^{[i]}$ for $i$-developed systems.

\section{Identity for $i$- and $j$- developed system.}

\begin{Theorem}
 Assume that  $\Delta=(\Delta_1,\dots,\Delta_{n+1})$ is $i$-developed and $j$-developed for some $i\neq j$. Assume also that $(f_1,\ldots,f_{n+1})\in \Omega_\Delta $ satisfies the assumptions of Theorem 9. Then
$$
\Pi_\Delta^{[i]}(f_1,\ldots,f_{n+1}) (\Pi_\Delta^{[j]}(f_1,\ldots,f_{n+1}))^{-1}=\prod [f_1,\ldots,f_{n+1}]_A^{(-1)^{n+i+j}k^{i,j}_A},
$$
where the product on the right is taken over the vertices $A$ of $\sum\Delta_i$.
\end{Theorem}
\begin{proof}
Let us apply the element $[f_1,\dots,f_{n+1}]\in H^n((\C^*)^n\setminus \Gamma,\C^*)$ to the identity from Theorem 9. According to the section 9.3 we  have

$$ \Pi_\Delta^{[i]}(f_1,\ldots,f_{n+1})^{(-1)^i (-1)^{j-2}} \Pi_\Delta^{[j]}(f_1,\ldots,f_{n+1}))^{(-1)^j(-1)^{i-1}} =$$
$$=\prod [f_1,\ldots,f_{n+1}]_A^{(-1)^nk^{i,j}_A}.$$

To complete the proof it is enough to raise each side of this identity to the power $(-1)^{i+j}$.
\end{proof}

Theorem 12 contains the formula from \cite{Kh2} for the product in $(\C^*)^n$ of all the roots of a developed system of $n$ equations. To find such a product it is enough to compute the product over all roots of any monomial $\mathbf z^{\mathbf m}$: taking the coordinate functions $z_1,\dots,z_n$ as such monomials one obtains all coordinates of the product of all roots. Assume that $\Delta=(\Delta_1,\dots,\Delta_{n+1})$ is, say, 1-developed and that $\Delta_1=\{m\}$ is a single point. Consider an $(n+1)$ tuple of Laurent polynomials $f_1,\dots,f_{n+1}$ with Newton polyhedra $\Delta_1,\dots, \Delta_{n+1}$.

Then: 1) $f_1$ is the monomial $\mathbf z^{\mathbf m}$ with a nonzero coefficient $c$, i.e. $f_1=c\mathbf z^{\mathbf m}$; 2) $\Delta$ is $j$-developed for any $1< j\leq (n+1)$ because the collection $\Delta$ with $\Delta_j$ skipped contains the point $\Delta_1$.

Let us apply Theorem 12 to the case under consideration with $i=1$, $j=2$. We have: a) $\Pi^{[2]}_\Delta=1$ because the system $cx^m=f_3=\dots=f_{n+1}=0$ has no roots in $(\C^*)^n$;
b) $\Pi^{[1]}_\Delta$ is equal to the product of $c\mathbf z^{\mathbf m}$ over all roots of the system $f_2=\dots=f_{n+1}=0$, i.e. is equal to $c^{n!Vol(\Delta_2,\dots,\Delta_{n+1})}$ multiplied by  the product of $\mathbf z^{\mathbf m}$ over all roots of the system.

\begin{Corollary}
With the assumptions of Theorem 12 for $i=1$, $j=2$ and $f_1=c \mathbf z^{\mathbf m}$ the product of $\mathbf z^{\mathbf m}$ over the roots of the system $f_2=\dots=f_{n+1}=0$ multiplied by $c^{n!Vol(\Delta_2,\dots,\Delta_{n+1})}$ is equal to $\prod [f_1,\ldots,f_{n+1}]_A^{(-1)^{n+1}k^{1,2}_A}.$
\end{Corollary}

\begin{Corollary}
If $\Delta=(\Delta_1,\dots,\Delta_{n+1})$ is $i$-developed and $j$-developed then on $\Omega_\Delta$  the relation
\begin{equation}\label {i,j}
\Pi_\Delta^{[i]}/\Pi_\Delta^{[j]}=M_{i,j}s_{i,j}
\end{equation}
holds, where $M_{i,j}$ is an explicit monomial in the vertex coefficients of all $f_k$ and $s_{i,j}=\prod_A(-1)^ {D(A_1,\ldots,A_{n+1})k^{i,j}_A}$, where $A_1,\dots,A_{n+1}$ are vertices of $\Delta_1,\dots,\Delta_{n+1}$ such that $A_1+\dots+A_{n+1}=A$.
\end{Corollary}

\begin{proof} By definition $ [f_1,\ldots,f_{n+1}]_A$ is an explicit monomial in coefficients of $f_1,\ldots,f_{n+1}$ corresponding to the vertices $A_1,\dots,A_{n+1}$ multiplied by $(-1)^{D(A_1,\ldots,A_{n+1})}$.
\end{proof}

\section{Identities for a completely developed system.}

\begin{Definition} A collection $\Delta=(\Delta_1,\dots,\Delta_{n+1})$  of  polyhedra is called {\it completely developed} if it is $i$-developed for all $1 \leq i\leq n+1$.
\end{Definition}

\begin{Theorem} For a completely developed  $\Delta$
there is a $(n+1)$-tuple  $(M_1,\dots, M_{n+1})$ where $M_k$ is a monomial depending on the vertex coefficients of $(f_1,\ldots,f_{n+1})\in \Omega_\Delta$ with $f_k$ removed and a $(n+1)$-tuple  $(s_1,\dots, s_{n+1})$ where $s_i=\pm 1$, such that:
\begin{equation}\label{PMs}
\Pi^{[1]}_\Delta M_1s_1=\ldots= \Pi^{[n+1]}_\Delta M_{n+1}s_{n+1}.
\end{equation}
The $(n+1)$-tuple  $(M_1,\dots,M_{n+1})$ of monials is unique and the $(n+1)$-tuple of signs $(s_1,\dots,s_{n+1})$ is unique up to simultaneous multiplication  by $-1$. Moreover the relation $s_is_j=s_{i,j}$ holds.
\end{Theorem}

\begin{proof}
To prove existence we will use the identities (\ref {i,j})  for $j>i=1$. Let us represent each monomial $M_{1,j}$ as a product $\prod_{1\leq k\leq n+1}m_{1,j}^{(k)}$ where $m_{1,j}^{(k)}$ is a monomial depending on the vertex coefficients of $f_k$ only. Denote  $m= \prod_{j\ne 1} m_{1,j}^{(j)}$ and divide each identity (\ref{i,j}) for $j>i=1$ by $m$. We obtain a needed representation with
$
(M_1,\dots, M_{n+1})=(m^{-1}, M_{1,2}m^{-1},\dots,M_{1,n+1}m^{-1}),
$ and
$
(s_1,\dots, s_{k+1})= (1,s_{1,2},\dots, s_{1,n+1}).
$

To show uniqueness assume that $(M_1',\dots,M_{n+1}')$ and $(s_1',\dots,s_{n+1}')$ are  another pair of $(n+1)$-tuples  of monomials and signs  such that:
$$
\Pi_\Delta^{[1]} M_1's_1'=\ldots= \Pi_\Delta^{[n+1]} M_{n+1}'s_{n+1}'.
$$
For any $i$ the ratio $M_i/M'_i$ is a  monomial which does not depend on coefficients of $f_i$ . But since $M_1s_1/M_1's_1' =\ldots= M_{n+1}s_{n+1}/M'_{n+1}s_{n+1}'$,
the ratio $M_i/M_i'$ is equal to 1 and collections of signs are proportional. The relation $s_is_j=s_{i,j}$ follows from (\ref {i,j}).
\end{proof}

\begin{Remark}
Let $G=(\Z/2\Z)^{n+1}/D$ be the factor group of $(\Z/2\Z)^{n+1}$ by the diagonal subgroup $D=\{(1,\ldots,1),(-1,\ldots,-1)\}$. Assigning to a collection of completely developed polyhedra $\Delta_1,\ldots,\Delta_{n+1}$ the collection of signs $(s_1,\dots, s_{n+1})$ defined up to simultaneous multiplication by $-1$  gives a map to $G$. This map is a coordinatewise homomorphism with respect to Minkowski sum, for example the relation
$$
\phi(\Delta_1+\Delta_1',,\ldots,\Delta_{n+1})= \phi(\Delta_1,\ldots,\Delta_{n+1})\phi(\Delta_1',\ldots,\Delta_{n+1}),
$$
holds for any completely developed collections $(\Delta_1,\Delta_2, \ldots,\Delta_{n+1})$ and $(\Delta_1',\Delta_2,\ldots,\Delta_{n+1})$.

The multihomomorphism $\phi$ is closely related to the resultants. Take any collection of Laurent polynomials $f=(f_1,\ldots, f_{n+1})$ with completely developed collection of Newton polyhedra  $\Delta=(\Delta_1,\ldots,\Delta_{n+1})$ such that all vertex coefficients of $f_i$'s are equal to 1. Then the values of the $n+1$ product resultants on the collection $f$ would coincide up to a sign, and so $\phi(\Delta)=(R^{[1]}_{\Pi\Delta},\ldots,R^{[n+1]}_{\Pi\Delta})$ up to multiplication by  common factor.

In a contrast to the $\Delta$-resultants, the map $\phi$ in general is not translation invariant (although, it is invariant under simultaneous translation of $\Delta_i$'s) and is not symmetric.
\end{Remark}

Theorems 12 and 13 allow  to describe the monomials $M_1,\dots,M_{n+1}$ in terms of Parshin symbols. Now we will describe these monomials in terms  of Newton polyhedra. Let us introduce some notation.

For an $i$-developed collection $\Delta_1,\dots, \Delta_{n+1}$ denote by $\tilde \Delta_i$  the sum $\Delta_1+\dots+\hat\Delta_i+\dots+\Delta_{n+1}$ where $\Delta_i$ is removed.  For each facet $\Gamma\subset\tilde \Delta_i$ denote by $v_\Gamma$ an irreducible integral covector such that the inner product with $v_\Gamma$ attains its maximum value on $\tilde \Delta_i$ at the facet $\Gamma$. With $v_\Gamma$ one associates the value  $H_{\Delta_i}(v_\Gamma)$ of the support function of $\Delta_i$ on $v_\Gamma$, the faces $\Delta_j^{v_\Gamma}$ of $\Delta_j$ at which the inner product with $v_\Gamma$ attains the maximal value. The facet $\Gamma$ is {\it essential} if among the faces $\Delta_j^{v_\Gamma}$ with $j\neq i$  exactly one face $\Delta_{j(v_{\Gamma})}^{v_\Gamma}$ is a vertex. With an essential facet $\Gamma$ one associates a coefficient $a_{j(v_\Gamma)}$ of the Laurent polynomial $f_{j(v_{\Gamma})}$ at the vertex $\Delta_{j(v_{\Gamma})}^{v_\Gamma}$, and the integral mixed volume $V(v_{\Gamma})$ of the collection of polyhedra $\{\Delta_j^{v_\Gamma}\}$ in which the polyhedra $\Delta_i^{v_\Gamma}$ and $\Delta_{j(v_\Gamma)}^{v_\Gamma}$ are removed.

Let $L(\Gamma)$ be a linear subspace parallel to the minimal affine subspace containing $\Gamma$. We define the integral volume on $L(\Gamma)$ as the translation invariant volume normalized by the following condition: for any $v_1,\ldots,v_{n-1}$ the generators of the lattice $L(\Gamma) \cap \Z^n$, the volume the parallelepiped with sides $v_1,\ldots,v_{n-1}$ is equal to 1.

Any polyhedron in the collection $\{\Delta_j^{v_\Gamma}\}$ in which the polyhedra $\Delta_i^{v_\Gamma}$ and $\Delta_{j(v_\Gamma)}^{v_\Gamma}$ are removed could be translated to $L(\Gamma)$. By $V(v_{\Gamma})$ we mean the integral mixed volume of these translations (note that $V(v_{\Gamma})$ could vanish for some $\Gamma$).

\begin{Theorem} For the monomial $M_i$ the following formula holds
$$ M_i=\prod a_{j(v_\Gamma)}^{(n-1)!H_{\Delta_i}(v_\Gamma)V(v_{\Gamma})}$$
where the product is taken over all essential facets $\Gamma$ of $\tilde \Delta_i$.
\end{Theorem}

\begin{proof} Let us sketch a proof for $M_1$. In Theorem 13 we represented $M_1$ in the form  $M_1=(\prod_{j\neq 1} m_{1,j}^{(j)})^{-1}$. One can deal with each factor $m_{1,j}^{(j)}$ separately. We will show that $m_{1,2}^{(2)}=\prod_{\Gamma} a_{2(v_\Gamma)}^{d(v_{\Gamma})}$, where $a_{2(v_\Gamma)}$ is the coefficient of $f_2$ at the vertex $\Delta_2^{v_{\Gamma}}$,
$
d(v_\Gamma)=(n-1)!V(v_\Gamma)H_{\Delta_1}(v_\Gamma) ,
$
and the product is taken over all facets $\Gamma$ of $\Delta_{1,2}=\Delta_3+\dots+\Delta_{n+1}$.

Let $C\subset (\C^*)^n$ be the curve defined by the system $f_3=\ldots=f_{n+1}=0$ (we assume that this system is generic enough). The normalization $\tilde C$ of $C$  has a very explicit description: it can be obtained as the closure of $C$ in the toric  comactification $X$ of $(\C^*)^n$ associated with the polyhedron $\Delta_{12}$. In particular, each facet $\Gamma$ of $\Delta_{12}$ corresponds to a codimension 1 orbit $X_\Gamma$ in $X$. The equality $\tilde C \setminus C=\bigcup_\Gamma(\tilde C\cap X_\Gamma)$ holds.  Moreover the number of points in $\tilde C \cap X_\Gamma$ is equal to $(n-1)! V(v_\Gamma)$, (see \cite{Kh} for details).

By (\ref {i,j}) we have $M_{1,2}=\pm\Pi_\Delta^{[1]}/\Pi_\Delta^{[2]}$. By definition  $M_{1,2}= m_{1,2}^{(1)}m_{1,2}^{(2)}F,$ where $F=\prod_{k>2} m_{1,2}^{(k)}$ is independent of $f_1,f_2$.
On the other hand $\Pi_\Delta^{[1]}/\Pi_\Delta^{[2]}$ is equal to the product of $\{f_1,f_2\}_p^{-1}$ over all zeros $p$ of $f_1f_2$ on the curve $C$.
By Weil's theorem this product is equal to
$
\prod_{q\in (\tilde C \setminus C)} \{f_1,f_2\}_q.
$

Explicit calculations show that for any $g\in \tilde C\cap X_\Gamma$ the following identity holds: $\{f_1,f_2\}_p = a_{1(v_\Gamma)}^{H_{\Delta_{2}}(v_\Gamma)} a_{2(v_\Gamma)}^{-H_{\Delta_{1}}(v_\Gamma)} G$, where $G$  is independent of $f_1,f_2$). The number of points in $\tilde C \cap X_F$ is equal to $(n-1)! Vol(\Delta_{3}^v,\ldots,\Delta_{n+1}^v)$. Putting everything together we get the needed identity $m_{1,2}^{(2)}=\prod_{\Gamma} a_{2(v_\Gamma)}^{d(v_{\Gamma})}$.
\end{proof}

\begin{Definition} For completely developed $(n+1)$-tuple $\Delta$ and $1\leq i\leq n+1$ we define the {\it $i$-th product resultant} $R^{[i]}_{\Pi\Delta}$ on $\Omega_\Delta$ as $R^{[i]}_{\Pi\Delta}=\Pi_\Delta^{[i]} M_i$. By (\ref{PMs}) all the product resultants $R^{[i]}_{\Pi\Delta}$  are equal up to sign.
\end{Definition}

\begin{Theorem}
Let $\Delta=(\Delta_1,\dots,\Delta_{n+1})$ be a completely developed collection. Then:

1) each product resultant $R^{[i]}_{\Pi\Delta}$  is a polynomial on $\Omega_\Delta$. The degree of  $R_{\Pi\Delta}^{[i]}$  in the coefficients of $f_j$ is equal to the  number of roots of the generic system $f_1=\ldots=f_{n+1}=0$ with  $f_j$ skipped (i.e is equal to $n!\,Vol(\Delta_1,\ldots, \hat \Delta_j,\ldots,\Delta_{n+1})$).

2) the function $R^{[i]}_{\Pi\Delta}$ is equal to zero at $(f_1,\ldots, f_{n+1})\in \Omega_\Delta^{[i]}$ if  and only if the system $f_1=\ldots =f_{n+1}=0$ has a root in $(\C^*)^n$.
\end{Theorem}

\begin{proof}
The expression $M_i\Pi_\Delta^{[i]}=R^{[i]}_{\Pi\Delta}$ is obviously  a polynomial of degree   $n!\,Vol(\Delta_1,\ldots, \hat \Delta_j,\ldots,\Delta_{n+1})$ in the coefficients of $f_i$.
Since all product resultants are equal up to sign we have proven 1).

Statement 2) is obvious from the definitions.
\end{proof}

Let $m\in \Delta_i\cap\Z^n$. Denote by $c$ the coefficient in front of $z^m$ in the Laurent polynomial $f_i$ with Newton polyhedron $\Delta_i$.

\begin{Theorem} In the notations from  Theorem 15, the degree of $R^{[j]}_{\Pi\Delta}$ in a specific coefficient $c$ of $f_i$ is equal to $n!\,Vol(\Delta_1,\ldots, \hat \Delta_i,\ldots,\Delta_{n+1})$. Each polynomial $R^{[j]}_{\Pi\Delta}$  contains exactly one monomial of the highest degree in $c$ and the coefficient in front of this monomial is $\pm 1$.
\end{Theorem}

\begin{proof} Without loss of generality we can assume that $i=1$. Since all product resultants are equal up to sign it is enough to prove the statement for $R^{[1]}_{\Pi\Delta}=M_1\Pi_\Delta^{[1]}$. The monomial $M_1$ is independent of the coefficients of $f_1$ and monomial of the highest degree in $c$ comes from multiplying the monomial $cz^m$ over roots of $f_2=\dots=f_{n+1}=0$. Now the theorem follows from Corollary 3.
\end{proof}

\section{Sums of Grothendieck residues over roots of developed system.}

In this section we discuss a formula from \cite{GKh}, \cite{GKh1} for the  sum of Grothendieck residues over the roots of a developed system. As a corollary we provide an algorithm for computing the product of values of a Laurent polynomial over the roots of a developed system.

\subsection{Grothendieck residue.}
 Consider the system  $f_1=\ldots=f_n=0$ in $(\C^*)^n$ and the hypersurface $\Gamma$ defined by $f_1\cdot\ldots\cdot f_n=0$. Let $\omega$ be a holomorphic $n$-form   on  $(\C^*)^n\setminus \Gamma$.
\begin{Definition}
 The  Grothendieck residue of $\omega$ at the root $z$ of the system  $f_1=\ldots=f_n=0$  is defined as the number $\frac{1}{(2\pi i)^n}\int_{\gamma_z}\omega$, where $\gamma_z$ is the Grothendieck cycle at $z$.
 \end{Definition}
As $\omega$ is automatically closed, the Grothendieck residue at the root $z$ is well defined.

\subsection{The residue of the form at a vertex of a polyhedron.} For each vertex $A$ of the Newton Polyhedron $\Delta(P)$ of a Laurent polynomial $P$, we will construct the Laurent series of the function $f/P$, for any Laurent polynomial $f$.

Let $q_A \neq 0$ be the coefficient of the  monomial in $P$ which corresponds to the vertex $A$ of $\Delta(P)$. The constant term of the Laurent polynomial $\tilde P = P/(q_A z^a)$ equals one. We will define the Laurent series of $1/\tilde P$ by the formula:
$$
1/\tilde P=1+ (1-\tilde P)+(1-\tilde P)^2+\ldots.
$$
Since each monomial $z^b$ appears only in finitely many summands $(1-\tilde P)^k$, the above sum is well defined. The {\it Laurent series of the rational function $f/P$ at the vertex $A$ of $\Delta(P)$} is the product of the series $1/\tilde P$ and the Laurent polynomial $q_A z^a f$.

Consider the $n$-form $\omega_f=fdz_1\wedge\ldots\wedge dz_n/Pz_1\cdot\ldots\cdot z_n$.
\begin{Definition} The  Grothendieck residue $res_A\omega$ of $\omega=\omega_f$ at the vertex  $A$ of $\Delta(P)$   is defined as the number $\frac{1}{(2\pi i)^n}\int_{T_A}\omega_f$, where $T_A$ is the cycle assigned to a vertex $A$ {\rm (see sec.7.2)}.
\end{Definition}

\begin{Lemma} The residue $res_A\omega_f$ is equal to the coefficient in from of the monomial $(z_1\cdot\ldots \cdot z_n)^{-1}$ in Laurent series of $f/P$ at the vertex $A$.
\end{Lemma}
We will not prove this simple lemma. See (\cite {GKh1}) for the details.

\subsection{Summation formula.} Consider the developed system of equations $f_1=\ldots=f_n=0$ in $(\C^*)^n$ with Newton polyhedra $\Delta_1,\dots,\Delta_n$. Denote by $P$ the product $f_1\cdot\ldots\cdot f_n$.
\begin{Theorem} For any Laurent polynomial $f$
the sum of the Grothendieck residues of the form $\omega_f=fdz_1\wedge\ldots\wedge dz_n/Pz_1\cdot\ldots\cdot z_n$ over all the roots of the system is equal to $(-1)^n\sum k_A res_A\omega_f$, where the summation is taken over all vertices $A$ of $\Delta_1+\ldots+\Delta_n$.
\end{Theorem}

\begin{proof} Theorem 17 follows from theorem 8 (see sections 12.1, 12.2).
\end{proof}

\begin{Corollary}
The sum $\sum f(z)\mu(z)$ of the values of any Laurent polynomial $f$ over all roots $z$ of a developed system, counted with multiplicities $\mu(z)$, is equal to $(-1)^n\sum k_A res_A\omega_\varphi$ where $\varphi= f\det M$ and $M$ is $(n\times n)$-matrix with the entries $M_{i,j}=\partial f_i/\partial z_j$.
\end{Corollary}

\begin{proof} The Grothendieck residue of  $\omega=fdf_1\wedge\ldots\wedge df_n/f_1\cdot\ldots\cdot f_n$ at  the root $z$ is equal to $f(z)\mu(z)$. It is easy to see that $\omega=\omega_\varphi$. So Corollary 5 follows from Theorem 17.
\end{proof}

\subsection{Elimination theory.} Here we use notations from the previous section.
Let $f$ be a Laurent polynomial. We will explain how to find any symmetric function of the sequence of the numbers $\{f(z)\}$ for all roots $z$ of the system (each root $z$ is taken with multiplicity $\mu(z)$)

Denote by $f^{[k]}$ the number
$
f^{[k]}=\sum_z f^k(z)\mu(z).
$
By Corollary 5  one can  calculate $f^{[k]}$ for any $k$ explicitly. The power sum symmetric polynomials form a generating set for the ring of symmetric polynomials.

\begin{Corollary} One can find explicitly all symmetric functions of $\{f(z)\}$ and construct a monic polynomial whose roots are $\{f(z)\}$.
In particular one can compute $\Pi_\Delta^{[i]}=\prod_z f(z)^{\mu(z)}$ for any $i$-developed system.
\end{Corollary}

\section{$\Delta$-Resultants.}

\subsection{Definition and some properties of $\Delta$-resultant.} Following \cite{GKZ} we  define the $\Delta$-resultant for a collection $\Delta=(\Delta_1,\ldots,\Delta_{n+1})$ of $(n+1)$ Newton polyhedra in  $\R^n$. We also generalize this definition to a collection $\Delta$ in $\R^N$ with $N \geq n$ such that the polyhedron $\Delta_1+\dots+\Delta_{n+1}$ has dimension  $\leq n$.

For $\Delta=(\Delta_1,\ldots,\Delta_{n+1})$ with $\Delta_i\subset \R^n$ denote by $X_\Delta\subset \Omega_\Delta$ the quasi-projective set of points $(f_1,\dots,f_{n+1})\in \Omega_{\Delta}$ such the system $f_1=\dots=f_{n+1}=0$ in $(\C^*)^n$ is consistent. The set $X_\Delta$ is irreducible, an easy proof of this fact can be found  in \cite{GKZ}.

\begin{Definition}
The {\it $\Delta$-resultant} $R_{\Delta}$ is a polynomial on $\Omega_{\Delta}$ satisfying   the following conditions:

$(i)$ The degree of  $R_{\Delta}$  in the coefficients of $f_i$ is equal to the  number of roots of the generic system $f_1=\ldots=f_{n+1}=0$ with  $f_i$ skipped (i.e is equal to $n!\,Vol(\Delta_1,\ldots, \hat \Delta_i,\ldots,\Delta_{n+1})$.) The coefficients of $R_\Delta$ are coprime integers.

$(ii)$ If the codimension of  $X_{\Delta}$ in $\Omega_\Delta$ is greater then 1 then $R_\Delta\equiv\pm 1$.

$(iii)$ $R_{\Delta}(f_1,\dots,f_{n+1})=0$ if and only if $(f_1,\dots,f_{n+1})$ belongs to the closure $\overline X_\Delta$ of $X_\Delta$ in $\Omega_\Delta$.
\end{Definition}

 \begin{Theorem} [\cite{GKZ}] \footnote{
In \cite{GKZ} the $A$-resultant is defined, the connection between $A$-resultants and $\Delta$-resultants is described in introduction.

In \cite{GKZ} $A$-resultant is defined under some assumption on $(n+1)$-tuple of supports $A$. The general definition of the $A$-resultant is given in \cite{Est}.

The condition $(i)$ on the degrees of resultant could be replaced by the condition that the resultant is a polynomial which vanishes on $X_\Delta$ with multiplicity equal to the generic number of solutions of consistent system (see \cite{D'AS}).}
There exists a unique up to sign polynomial $R_{\Delta}$  on $\Omega_\Delta$ satisfying the conditions $(i)-(iii)$.  \end{Theorem}

The $\Delta$-resultant obviously has the following properties: 1) it is independent of the ordering of the polyhedra from the set $\Delta=(\Delta_1,\dots,\Delta_{n+1})$;
2) it is invariant under translations of the polyhedra from the set  $\Delta$; 3) it is invariant under linear transformations of $\R^n$ inducing an automorphism of the lattice $\Z^n\subset \R^n$.

\begin{ex}\label{resmon}
Suppose $\Delta_1=\{m\}$ is a one point set, i.e. $f_1=cz^m$ is a monomial $z^m$ with some coefficient $c$. Then
$R_{\Delta}= \pm c^{n!\,Vol(\Delta_2,\ldots,\Delta_{n+1})}$.

\end{ex}

Assume that $\Delta=(\Delta_1,\dots,\Delta_{n+1})$ with $\Delta_i\subset \R^N$ satisfying inequality $\dim (\Delta_1+\dots+\Delta_{n+1})\leq n$.  Choose any linear isomorphism $A:\R^N\rightarrow \R^N$ preserving the lattice $\Z^N$ and choose vectors $v_1,\dots,v_{n+1}\in \Z^N$ such that the polyhedra $\Delta'_i=A(\Delta_i)+v_i$ belong to the $n$ dimensional coordinate subspace $\R^n\subset \R^N$.

\begin{Definition} The generalized $\Delta$-resultant for $\Delta$ as above is defined as the $\Delta'$-resultant for $\Delta'=(\Delta'_1,\ldots,\Delta'_{n+1})$. The generalized $\Delta$-resultant is well defined, i.e. is independent of the choice of $A$ and $v_1,\dots,v_{n+1}$.

\end{Definition}

\subsection {Product resultants  and $\Delta$-resultant} In this section we show that for a completely developed collection $\Delta$ all product resultants are equal up to sign to the $\Delta$-resultant.

The Sylvester formula represents the $\Delta$-resultants of two polynomials in one variable as a determinant of one of two explicitly written matrices (see sec. 2.1).  In \cite{CE} the Sylvester formula was beautifully generalized in the following way. For the collection of $n+1$ Newton polyhedra $\Delta$, Canny and Emiris construct $n+1$ matrices $M_i$ (which coincide with Sylvester's matrix up to permutation of rows in dimension 1) such that the $\Delta$-resultant $R_\Delta$ divides their determinants. Moreover, $R_\Delta$ is the greatest common divisor of polynomials $det(M_i)$, thus they obtain a  practical algorithm for computing $R_\Delta$. The construction in \cite{CE} heavily uses geometry of Newton polyhedra.

By the extreme monomials of a polynomial $P$ we will mean the monomials corresponding to the vertices of the Newton polyhedron  of $P$. In \cite{St} Sturmfelds generalized the construction in  \cite{CE}  and using this generalization proved the following theorem.

\begin{Theorem}
All extreme monomials of the $\Delta$-resultant have coefficient $-1$ or $+1$.
\end{Theorem}

Now we are able to prove the following theorem.

\begin{Theorem} For a completely developed $\Delta$ for any  $1\leq i\leq n+1$ the product resultant $R^{[i]}_{\Pi\Delta}$ is equal up to sign to the $\Delta$-resultant $R_\Delta$.
\end{Theorem}
\begin{proof} Without the loss of generality we can assume that $i=1$.
Both functions $R_{\Delta}$ and $R_{\Pi,\Delta}^{[1]}$ are polynomials in the coefficients of $f_1,\dots,f_{n+1}$ of the same degrees  and they both vanish on the set $\overline X_\Delta$ (see Theorem 15). Since the set $\overline X_\Delta$ is irreducible polynomials $R_{\Delta}$ and $R_{\Pi,\Delta}^{[1]}$ are proportional. According to Corollary 3 for any chosen  coefficient $c$ of $f_1$, in the polynomial $R_{\Pi\Delta}^{[1]}$ there is a unique monomial  having the highest degree in $c$ and the coefficient in $R_{\Pi\Delta}^{[1]}$ in front of this monomial is $\pm 1$. But according to Theorem 19 the coefficient in $R_\Delta$ in front of this monomial is also $\pm 1$. So $R_{\Delta}=\pm R_{\Pi,\Delta}^{[1]}$.
\end{proof}

\subsection {The Poisson formula.}  The inductive Poisson formula for $\Delta$-resultant (see \cite{PSt}, \cite{D'AS}) and the summation formula (see section 12.3) allow to provide an algorithm computing the $\Delta$-resultant for 1-developed systems. To state the formula  let us define all terms appearing in it. Let $\Delta$ be $(\Delta_1,\dots, \Delta_{n+1})$ and let $(f_1,\dots, f_{n+1})$ be a point in $\Omega_\Delta$. The only term in the formula  depending on the coefficients of $f_1$ is the term  $\Pi_\Delta^{[1]}$. To present the other terms we need some  notation.

Denote by $\tilde \Delta_1$  the sum $\Delta_2+\dots+\Delta_{n+1}$.  For each facet $\Gamma$ of $\tilde \Delta_1$  denote by $v_\Gamma$ the irreducible integral covector such that the inner product of $x\in \tilde \Delta_1$ with $v_\Gamma$ attains its maximum value on  $\Gamma$. With $v_\Gamma$ one associates the value  $H_{\Delta_1}(v_\Gamma)$ of the support function of $\Delta_1$ on $v_\Gamma$, the faces $\Delta_j^{v_\Gamma}$ of $\Delta_j$ on which the inner product of $x\in \Delta_j$ with $v_\Gamma$ attains its maximal value. By $f_j^{v_\Gamma}$ we denote the sum  $\sum c_mx^m$ over $m\in \Delta_j^{v_\Gamma}\cap \Z^n$, where $c_m$ is the coefficient in front of $x^m$ in $f_j$. For each $v_\Gamma$ the collection of $n$ polyhedra $\tilde \Delta_1^{v_\Gamma}=(\Delta_2^{v_\Gamma},\dots,\Delta_{n+1}^{v_\Gamma})$ satisfies the inequality $\dim (\Delta_2^{v_\Gamma}+\dots+\Delta_{n+1}^{v_\Gamma})\leq (n-1)$. This is why the generalized $\tilde \Delta_1^{v_\Gamma}$ resultant $R_{\tilde \Delta_1^{v_\Gamma}}(f_2^{v_\Gamma},\dots,f_{n+1}^{v_\Gamma})$ is defined.
Now we are ready to state the Poisson formula.

\begin{Theorem} The following Poisson formula holds:
$$R_{\Delta}(f_1,\dots,f_{n+1})=\pm
\Pi_{\Delta}^{[1]}(f_1,\dots,f_{n+1})\prod R_{\tilde \Delta_1^{v_\Gamma}}^{H_1(v_\Gamma)}(f_2^{v_\Gamma},\dots,f_{n+1}^{v_\Gamma}),
$$ where the product is taken over all facets $\Gamma$ of $\tilde\Delta_1$.
\end{Theorem}

In a subsequent paper we are going to give an elementary proof of Theorem 21 (in fact, we will generalize the Poisson formula from the toric case to a larger class of algebraic varieties).

\begin{Remark} Mixed volume and $\Delta$-resultants have many similar properties. For example the non symmetric formula for the mixed volume
$
Vol(\Delta_1,\ldots,\Delta_n)=\frac{1}{n}\sum_v H_1(v)Vol(\Delta_2^v,\ldots,\Delta_n^v)
$
is analogues to the Poisson formula for the $\Delta$-resultants.

\end{Remark}

 For a 1-developed collection $\Delta=(\Delta_1,\ldots,\Delta_{n+1})$ of Newton polyhedra the Poisson formula becomes much simpler: in this case the resultants $R_{\tilde \Delta_1^{v_\Gamma}}(f_1^{v_\Gamma},\dots,f_{n+1}^{v_\Gamma})$ can be computed explicitly. Below we present  such computation.

 By definition of $\Delta$ being  1-developed collection, for any facet $\Gamma$ of $\tilde \Delta_1$ in the set  $\{\Delta_j^{v_\Gamma}\}$ with $j>1$ at least one polyhedron $\Delta_{j(v_\Gamma)}^{v_\Gamma}$ is a point (if more then one polyhedron is a point denote by $\Delta_{j(v_\Gamma)}^{v_\Gamma}$ any of them).  Denote by $Vol(v_\Gamma)$ the integral $(n-1)$ dimensional mixed volume  of collection  $\{\Delta_j^{v_\Gamma}\}$ with $j>1$ in which the polyhedron  $\Delta_{j(v_\Gamma)}^{v_\Gamma}$ is skipped. Denote by $a_{j(v_\Gamma)}$ the coefficient of the Laurent polynomial $f_{j(v_{\Gamma})}$ at the vertex $\Delta_{j(v_{\Gamma})}^{v_\Gamma}$. In the above notation Example~\ref{resmon}  provides us the formula:
$$
R_{\tilde\Delta_1}^{v_\Gamma}(f_2^{v_\Gamma},\dots,f_{n+1}^{v_\Gamma}) = \pm a_{j(v_\Gamma)}^{(n-1)!\,Vol(v_\Gamma)}.
$$

\begin{Corollary}
With the notation as above the $\Delta$-resultant of the  1-developed collection $\Delta$ is given by:
$$
R_{\Delta}(f_1,\dots,f_{n+1})=\pm \Pi^{[1]}_\Delta (f_1,\dots,f_{n+1})\prod_\Gamma a_{j(v_\Gamma)}^{(n-1)!\,Vol(v_\Gamma)H_1(v_\Gamma)}.
$$
\end{Corollary}

\begin{Corollary} Using Corollary 6 one can produce an algorithm for computing the $\Delta$-resultant of a 1-developed collection $\Delta$.
\end{Corollary}

Indeed, the only implicit term in the formula from the Corollary 7 is the term $\Pi_\Delta^{[1]}$. Corollary 6 provides an algorithm for its computation.

\subsection {A sign version of Poisson formula}

Let $\Delta$ be a developed collection. According to Theorem 13 with such $\Delta$, an  $(n+1)$-tuple of monomials $M_1,\dots, M_{n+1}$ and an $(n+1)$-tuple of signs are defined.
According to the formula from Corollary 7 the Poisson formula in that case can be written as $R_\Delta=\pm \Pi_\Delta^{[1]}M_1$. By definition $\Delta$ is not only 1-developed, it is $i$-developed for any $i$. Thus one can write the formula from Corollary 7 putting instead of $f_1$ any $f_i$.

\begin{Corollary} The following equalities hold:
$$
\pm \Pi_\Delta^{[1]}M_1=\dots=\pm \Pi_\Delta^{[n+1]}M_{n+1}=\pm R_\Delta.
$$

\end{Corollary}

Thus from the theory of $\Delta$-resultants one can prove the product identities from Theorem 13 up to sign. It is impossible to reconstruct the signs in these identities using $\Delta$-resultants: the $\Delta$-resultant itself is defined up to sign only. Our Theorem 13 provided the sign version
$$
 \Pi_\Delta^{[1]}M_1s_1=\dots=\Pi_\Delta^{[n+1]}M_{n+1}s_{n+1}
 $$
of Poisson type identities and Theorem 20 provides identities $\Pi_\Delta^{[i]}M_i=\pm R_\Delta$ (which unavoidable could be up to sign only).



\begin{thebibliography}{99}
\nopagebreak[3]


\bibitem[B]{B}
Bernshtein, David N. "The number of roots of a system of equations." Functional Analysis and its applications 9.3 (1975): 183-185.


\bibitem[CE]{CE}
Canny, J., \& Emiris, I. (1993). An efficient algorithm for the sparse mixed resultant. In International Symposium on Applied Algebra, Algebraic Algorithms, and Error-Correcting Codes (pp. 89-104). Springer Berlin Heidelberg.


\bibitem[D'AS]{D'AS}
D'Andrea, C., and Sombra, M. (2015). A Poisson formula for the sparse resultant. Proceedings of the London Mathematical Society, pdu069.


\bibitem[Est]{Est}
Esterov, A. (2010). Newton polyhedra of discriminants of projections. Discrete \& Computational Geometry, 44(1), 96-148.


\bibitem[FPar]{FPar}
Fimmel T. and Parshin A. N., ÓIntroduction to higher adelic theoryÓ, Preprint, Steklov Mathematical Institute, Moscow, 1999


\bibitem[GKZ]{GKZ}
Gelfand, I. M., Kapranov, M.,  Zelevinsky, A. (2008). Discriminants, resultants, and multidimensional determinants. Springer Science \& Business Media.


\bibitem[GKh]{GKh}
Gelfond, O. A.,  Khovanskii, A. G. (1996). Newtonian polyhedrons and Grothendieck residues. In Doklady Akademii Nauk (Vol. 54, pp. 700-702).


\bibitem[GKh1]{GKh1}
Gelfond, Olga, and Askold Khovanskii. "Toric geometry and Grothendieck residues." Mosc. Math. J 2.1 (2002): 99-112.


\bibitem[Kh]{Kh}
Khovanskii, Askold G. Algebra and mixed volumes. In book Y.D. Burago and V.A. Zalgaller, Geometric inequalities, Springer-Verlag, Berlin and New York. V. 285, 182Ð207, 1988.


\bibitem[Kh1]{Kh1}
Khovanskii, Askold G. "Newton polygons, curves on torus surfaces, and the converse Weil theorem." Russian Mathematical Surveys 52.6 (1997): 1251-1279.

\bibitem[Kh2]{Kh2}
Khovanskii, Askold G. "Newton polyhedra, a new formula for mixed volume, product of roots of a system of equations." Fields Inst. Comm 24 (1999): 325-364.


\bibitem[Kh3]{Kh3}
Khovanskii, A. G. (2006). An analog of determinant related to ParshinÑKato theory and integer polytopes. Functional Analysis and Its Applications, 40(2), 126-133.


\bibitem[Kh4]{Kh4}
Khovanskii, A. (2006). Logarithmic functional and the Weil reciprocity laws. In Proceedings of the Waterloo Workshop on Computer Algebra (pp. 85-108).


\bibitem[Kh5]{Kh5}
Khovanskii, A. (2008). Logarithmic functional and reciprocity laws. Contemporary Mathematics, Vol. 460, 221-229.


\bibitem[Par]{Par}
Parshin A. N., Galois cohomology and Brauer group of local fields, Trudy Mat. Inst. Steklov, vol. 183, 1990; English Transl. Proc. Steklov Inst. Math., vol 183, 191Ð201, 1991.


\bibitem[PSt]{PSt}
Pedersen, Paul, and Bernd Sturmfels. "Product formulas for resultants and Chow forms." Mathematische Zeitschrift 214.1 (1993): 377-396.


\bibitem[Sop]{Sop}
Soprounov, Ivan. "Residues and tame symbols on toroidal varieties." Compositio Mathematica 140.06 (2004): 1593-1613.


\bibitem[St]{St}
Sturmfels, B. (1994). On the Newton polytope of the resultant. Journal of Algebraic Combinatorics, 3(2), 207-236.

\end{thebibliography}
\end{document}